\documentclass[oneside,english,reqno]{amsart}
\title{\textbf{Variance of resistance of "line-circle-line" graphs}}
\author{Alon Ivtsan}
\address{Alon Ivtsan,
Department of Mathematics, Ziskind Building, 
Weizmann Institute of Science, Rehovot 76100, Israel}
\email{aloniv@weizmann.ac.il}
\date{}
\usepackage{amsmath}
\DeclareMathOperator{\Var}{Var}

\usepackage{amssymb}
\usepackage{amsthm}
\usepackage[retainorgcmds]{IEEEtrantools}
\usepackage{graphicx}
\usepackage{bbm}
\newtheorem{thm}{Theorem}
\newtheorem{lem}{Lemma}

\usepackage{tikz}
\tikzstyle{vertex}=[circle, draw, inner sep=0pt, minimum size=6pt]

\usepackage{caption}
\usetikzlibrary{patterns,positioning,decorations.pathmorphing,decorations.pathreplacing}
\usetikzlibrary{arrows,automata,snakes}

\begin{document}
\maketitle

\begin{abstract}
We find 
the order of the variance of the growth model $X_{n+1}=X_n+X'_n+f\left(X''_n,X'''_n\right)$, where all the variables $X_n,X'_n,X''_n$ and $X'''_n$ are i.i.d., $X_0$ takes the values $1$ and $2$ with equal probability and $f$ is positive, monotone non-decreasing and satisfies conditions which, roughly speaking, pertain to its first and second order partial derivatives. For an appropriate choice of $f$ we obtain that the variance of the effective resistance between the endpoints of the "line-circle-line" graph $G_n$ is of order $\left(2+\frac{1}{8}+{\scriptscriptstyle\mathcal{O}}\left(1\right)\right)^n$. 
\end{abstract}

\bigskip
\bigskip

We define a sequence of recursively-defined graphs $\{G_i\}$ which we call the "line-circle-line"  sequence as follows: $G_1$ is pictured below, and $G_{k+1}$ is obtained by replacing each edge of $G_k$ by $G_1$ or, alternatively, by connecting three blocks in series: the first and last consist of $G_k$ and the middle block is two copies of $G_k$ connected in parallel. 

\bigskip
\bigskip

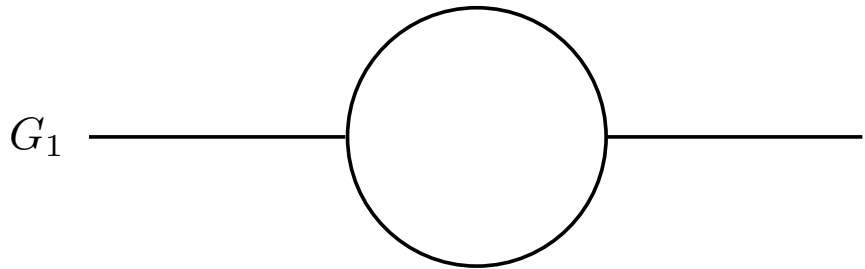
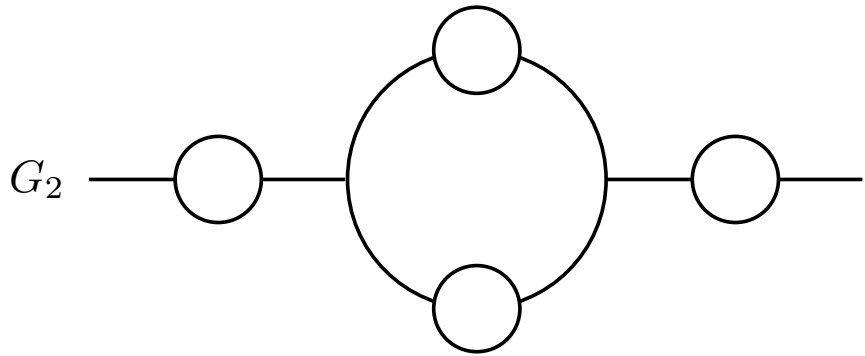
\resizebox{12cm}{!}{
\begin{tikzpicture}[-,>=stealth',shorten >=0.5pt,auto,node distance=2cm,
 thick,main node/.style={circle,fill=blue!10,draw,font=\sffamily\large\bfseries}]
 
\draw (0,0) circle (1);
\draw [-] (-3,0) -- (-1,0);
\draw [-] (1,0) -- (3,0);
\node at (-3.4,0) {$G_1$};

\draw (0,-3) circle (1);
\draw [-] (-3,-3) -- (-1,-3);
\draw [-] (1,-3) -- (3,-3);
\node at (-3.4,-3) {$G_2$};
\filldraw [fill=white,draw=black] (-2,-3) circle (1/3);
\filldraw [fill=white,draw=black] (2,-3) circle (1/3);
\filldraw [fill=white,draw=black] (0,-2) circle (1/3);
\filldraw [fill=white,draw=black] (0,-4) circle (1/3);


\end{tikzpicture}
}

\bigskip
\bigskip
 

In our theorem, one chooses a positive monotone non-decreasing (with respect to the partial order in $\mathbb R^2$) function $f$ which, slightly informally, satisfies technical conditions pertaining to its first and second order partial derivatives and one observes the growth model $X_{n+1}=X_n+X'_n+f\left(X''_n,X'''_n\right)$, where all the variables $X_n,X'_n,X''_n$ and $X'''_n$ are independent and identically distributed and $X_0$ takes the values $1$ and $2$ with equal probability. Our theorem tells us the order of the variance of $X_n$ and shows that a renormalized version of it converges to a normal random variable. If we select $f(t,s)=\frac{ts}{t+s}$, then $X_n$ denotes the effective resistance between the endpoints of $G_n$ (henceforth the resistance of $G_n$), and thus our theorem tells us that the variance of the resistance of $G_n$ is of order $\left(2+\frac{1}{8}+{\scriptscriptstyle\mathcal{O}}\left(1\right)\right)^n$. This result can be viewed as a special case of superconcentration as defined by Chaterjee \cite{chatterjee}. 
Probability models on similar recursively defined graphs have been studied before, see e.g. works by Hambley and Kumagai \cite{hambley-kumagai} and by Khristoforov, Kleptsyn and Triestino \cite{khristoforov}.  
The variance of the effective resistance on a different graph, the $d$-dimensional Torus, can be found in \cite{rossignol}.  We finish this paper by analysing another choice of $f$, this time $f(t,s)=t^{\frac{1}{2}}s^{\frac{1}{2}}$, since it satisfies a variant of one of the conditions of our theorem but does not satisfy the original condition of the theorem. 
 
\begin{thm}
Let $X_0$ be a random variable which takes the values $1$ and $2$ with probability one half and let $X_{n+1}$ be defined via the recurrence relation
$X_{n+1}=X_n+X'_n+f\left(X''_n,X'''_n\right)$, where all the variables $X_n,X'_n,X''_n$ and $X'''_n$ are independent and identically distributed and $f$ is a positive function on $\left[1,\infty\right)\times \left[1,\infty\right)$ with second order partial derivatives at every point which is also monotone non-decreasing in the following sense: $f\left(a,b\right)\leq f\left(c,d\right)$ whenever $a\leq c$ and $b \leq d$. Assume that the following conditions hold:
\begin{enumerate}
\item $\frac{\partial f}{\partial i}\bigg|_{\left(t,t\right)}\xrightarrow{t\to \infty} C_i$ for $i\in\{x,y\}$ and $0< C_x+C_y$ and $f(t,t)=\left(C_x+C_y\right)t$.
\item $\left(f(a_1,a_2)-f(a_3,a_4)\right)^2\leq A\left(a_1-a_3\right)^2+B\left(a_2-a_4\right)^2$ holds for $A,B\geq0$ which satisfy $A+B< C_x+C_y$ and for all $\left(a_1,a_2\right)$ and $\left(a_3,a_4\right)$ belonging to the domain $\{(s,t)|\left(2+C_x+C_y\right)^n \leq s,t \leq 2\left(2+C_x+C_y\right)^n\}$ for $n$ sufficiently large.
\item 
$\left(2+C_x+C_y\right)^{2n}\sup g^2={\scriptscriptstyle\mathcal{O}}\left(2^n\right)$ holds for $g\in\{\frac{\partial^2 f}{\partial x^2},\frac{\partial^2 f}{\partial y^2},\frac{\partial^2 f}{\partial x\partial y}\}$ and $n$ sufficiently large, where the supremum is taken over the same domain as in condition 2. 

\end{enumerate}
Then 
\begin{enumerate}
\item$\Var \left[X_n\right]=\left(2+C^2_x+C^2_y+{\scriptscriptstyle\mathcal{O}}\left(1\right)\right)^n$.
\item $\frac{X_n-\mathbb E\left[X_n\right]}{\sqrt{\Var \left[X_n\right]}}$ converges in probability to a standard normal random variable.
\item If in addition $f$ is either concave or convex then $\frac{\mathbb E\left[X_n\right]}{\left(2+C_x+C_y\right)^n}$ converges.
\end{enumerate}
\end{thm}
Some remarks:
\begin{enumerate}
\item Note that if $f(t,t)=Ct$, then clearly $C=C_x+C_y$, since if we write $g(t)=f(t,t)$ then $C=g'(t)=\frac{\partial f}{\partial x}\bigg|_{\left(t,t\right)}+\frac{\partial f}{\partial y}\bigg|_{\left(t,t\right)}\xrightarrow{t\to \infty} C_x+C_y\,$.
\item $C_x+C_y<1$ follows from condition 2. If we write $g(t)=f(t,t)$ then $\left(g\left(t\right)-g\left(\frac{t}{2}\right)\right)^2=\left(g'\left(s\left(t\right)\right)\right)^2\frac{t^2}{4}$ for some $s(t) \in \left[\frac{t}{2},t\right]$ and thus, by condition 2, $A+B\geq \left(g'\left(s\left(t_k\right)\right)\right)^2\xrightarrow{t_k\to \infty}\left(C_x+C_y\right)^2$ (we replace $t$ by a sequence $t_k$ which satisfies the additional domain restriction). Since by condition 2 $A+B<C_x+C_y\,$, we obtain that $C_x+C_y<1\,$.
\item If we can find constants $A$ and $B$ which satisfy $\left(f\left(a_1,a_2\right)-f\left(a_3,a_4\right)\right)^2\leq A\left(a_1-a_3\right)^2+B\left(a_2-a_4\right)^2$ but not $A+B< C_x+C_y\,$,  then condition 2 holds for $\varepsilon f$ for $\varepsilon >0$ small enough, so we can apply this theorem for $\varepsilon f$ instead of for $f$. 
\item Let $A_1$ and $B_1$ be new constants for which the inequality in condition 2 holds after taking expectation on both sides, where $a,b,c$ and $d$ are replaced by independent copies of $X_n$. Then (as can be seen in the proof) the result will still hold if we replace in condition 2 the condition $A+B<C_x+C_y$ by the conditions $A_1+B_1<C_x+C_y$ and $2+A^2+B^2<\left(2+A_1+B_1\right)^2$. We show later that this comment can be applied to the function $f(t,s)=t^{\frac{1}{2}}s^{\frac{1}{2}}$.
\item An analogous result for a function which satisfies all the above conditions apart from $f(t,t)=Ct$ can also be obtained. We can also prove the main theorem if we replace $f$ by $f_1+f_2\,$, where $f_1$ satisfies all the conditions $f$ satisfies and $f_2$ is positive and bounded.
\end{enumerate}
Examples of functions $f(t,s)$ satisfying the conditions: $\frac{ts}{t+s}$ and $c_p\left(t^p+s^p\right)^{\frac{1}{p}}$ for $p>0$ for some $c_p>0$ and $c_{\alpha}t^\alpha s^{1-\alpha}$ for $\alpha \in (0,1)$ and $c_{\alpha}>0\,$. An example which satisfies the second part of remark $5$ which is also monotone non-decreasing is $f(t,s)=\frac{1}{3}t+\frac{1}{3}s+\frac{1}{3}\sin^2(t-s)\,$.
\begin{proof}
We denote by $C_i(A,B)$ positive constants that depend only upon $A$ and $B$ and $C$ will denote an absolute constant whose value might change from line to line. From the recurrence relation of $X_n$, clearly $2\Var \left[X_n\right]\leq\Var \left[X_{n+1}\right]$ and thus $\Var \left[X_n\right]\geq C 2^n$. By applying condition 2, we obtain that for $n$ sufficiently large:
\begin{IEEEeqnarray*}{rCl}
\mathbb E\left[\left(X_{n+1}-X'_{n+1}\right)^4\right]&\leq&  \left(2+A^2+B^2\right)\mathbb E\left[\left(X_n-X'_n\right)^4\right]+C_0(A,B)\left(\Var\left[X_n\right]\right)^2 
\end{IEEEeqnarray*}
Since $(2+A^2+B^2)<2^2$ we obtain that:
\begin{IEEEeqnarray}{rCl}\label{eq:fourth_moment_0}
\mathbb E\left[\left(X_{n+1}-X'_{n+1}\right)^4\right]&\leq&  C_1(A,B)\left(\Var\left[X_n\right]\right)^2 
\end{IEEEeqnarray}
Since $f$ is monotone, the minimum and maximum values of $X_n$ are obtained by setting initial conditions $X_0=1$ and $X_0=2$ respectively, and thus, since $f(t,t)=\left(C_x+C_y\right)t\,$, $X_n$ lies between $\left(2+C_x+C_y\right)^n$ and $2\left(2+C_x+C_y\right)^n$. 
By applying condition 2, we obtain that for $n$ sufficiently large:
\begin{IEEEeqnarray*}{rCl}
\mathbb E\left[\left(X_{n+1}-X'_{n+1}\right)^2\right]\leq (2+A+B)\mathbb E\left[\left(X_n-X'_n\right)^2\right],
\end{IEEEeqnarray*}
and thus, since $A+B< C_x+C_y\,$, 
\begin{IEEEeqnarray*}{rCl}
\Var \left[X_n\right]\leq C\left(2+A+B\right)^n={\scriptscriptstyle\mathcal{O}}\left(\mathbb E \left[X_n\right]\right).
\end{IEEEeqnarray*}
Plugging this into equation \ref{eq:fourth_moment_0} yields:
\begin{IEEEeqnarray}{rCl}\label{eq:fourth_moment}
\mathbb E\left[\left(X_{n+1}-X'_{n+1}\right)^4\right]\leq C_2 (A,B)\left(2+A+B\right)^{2n}={\scriptscriptstyle\mathcal{O}}\left(\left(\mathbb E\left[X_n\right]\right)^2\right).\label{eqn:star}
\end{IEEEeqnarray}
Now, in order to simplify notation, let $a_1,a_2,a_3$ and $a_4$ be i.i.d. random variables distributed as $X_n$. Then
\begin{IEEEeqnarray}{rCl}\label{eq:variance_formula}
\Var \left[X_{n+1}\right]=2\Var \left[X_n\right]+\frac{1}{2}\mathbb E\left[\left(f(a_1,a_2)-f(a_3,a_4)\right)^2\right].
\end{IEEEeqnarray}
We estimate the second term using Taylor expansion:

\begin{IEEEeqnarray*}{rCl}
f(a_1,a_2)-f(a_3,a_4)&=&\frac{\partial f}{\partial x}\bigg|_{(b_1,b_2)}(a_1-a_3)+\frac{\partial f}{\partial y}\bigg|_{(b_1,b_2)}(a_2-a_4)\\
&=& \left(\frac{\partial f}{\partial x}\bigg|_{(b_1,b_2)}-C_x\right)(a_1-a_3)+C_x(a_1-a_3)\\
&& \text{ }+\text{ }\left(\frac{\partial f}{\partial y}\bigg|_{(b_1,b_2)}-C_y\right)(a_2-a_4)+C_y(a_2-a_4)\,.
\end{IEEEeqnarray*}

(here $(b_1,b_2)$ belongs to the convex hull of $(a_1,a_2)$ and $(a_3,a_4)\,$.) Thus, if we use the notation $A_1=\left(\frac{\partial f}{\partial x}\bigg|_{(b_1,b_2)}-C_x\right)$, $A_2=\left(\frac{\partial f}{\partial y}\bigg|_{(b_1,b_2)}-C_y\right)$, $B_1=C_x\,$, $B_2=C_y\,$, $c_1=a_1-a_3$ and $c_2=a_2-a_4$ then:

\begin{IEEEeqnarray}{rCl}\label{eq:variance_of_f}
\frac{1}{2}\mathbb E\left[\left(f(a_1,a_2)-f(a_3,a_4)\right)^2\right] &=& \left(C^2_x+C^2_y\right)\Var \left[X_n\right]\\
& & \text{ }+\text{ }\frac{1}{2}\sum_{i=1}^2\mathbb E\left[A_i^2c_i^2\right]+\sum_{i=1}^2\mathbb E\left[A_iB_ic_i^2\right] \nonumber \\
& & \text{ }+\text{ } \mathbb E\left[A_1B_2c_1c_2\right]+\mathbb E\left[A_2B_1c_1c_2\right] \nonumber \\
& & \text{ }+\text{ } \mathbb E\left[A_1A_2c_1c_2\right]. \nonumber
\end{IEEEeqnarray}

Now we show that all the terms on the right hand side of the previous equation aside from the first one are ${\scriptscriptstyle\mathcal{O}}\left(1\right)\Var \left[X_n\right]$. We bound the second term. Bounding the other terms is similar, and will follow from perhaps also applying Cauchy-Schwarz inequality. Let $\mu=\mathbb E\left[X_n\right]\,$. Then,
\begin{IEEEeqnarray*}{rCl}
\mathbb E\left[A_1^2c_1^2\right] &\leq& 2\mathbb E\left[\left(\frac{\partial f}{\partial x}\bigg|_{(b_1,b_2)}-\frac{\partial f}{\partial x}\bigg|_{(\mu,\mu)}\right)^2(a_1-a_3)^2\right]\\
&& { }+\;2\left(\frac{\partial f}{\partial x}\bigg|_{(\mu,\mu)}-C_x\right)^2\mathbb E\left[(a_1-a_3)^2\right].
\end{IEEEeqnarray*} 
The second term on the right hand side is ${\scriptscriptstyle\mathcal{O}}\left(1\right)\Var \left[X_n\right]$ by condition 1. Thus, it suffices to bound the first term on the right hand side. By Taylor's expansion:
\begin{IEEEeqnarray*}{rCl}
\frac{\partial f}{\partial x}\bigg|_{(b_1,b_2)}-\frac{\partial f}{\partial x}\bigg|_{(\mu,\mu)} &=& \frac{\partial^2 f}{\partial x^2}\bigg|_{(d_1,d_2)}(b_1-\mu)+\frac{\partial^2 f}{\partial x \partial y}\bigg|_{(d_1,d_2)}(b_2-\mu)\,.
\end{IEEEeqnarray*}  
(here $(d_1,d_2)$ belonging to the convex hull of $(b_1,b_2)$ and $(\mu,\mu)\,$.) Since $b_i$ is a convex combination of $a_i$ and $a_{i+2}$ for $i=1,2$ the inequality $(b_i-\mu)^2 \leq (a_i-\mu)^2+(a_{i+2}-\mu)^2$ holds. To simplify notation, let $A=\left(2+C_x+C_y\right)^n$, $B=\frac{\partial f}{\partial x}\bigg|_{(b_1,b_2)}-\frac{\partial f}{\partial x}\bigg|_{(\mu,\mu)}$ and $c_1=a_1-a_3$. Thus, by Cauchy-Schwarz, condition 3 and the inequality 
\begin{IEEEeqnarray}{rCl}\label{4-moment-easy}
2\mathbb E\left[\left(X_n-\mu\right)^4\right]\leq \mathbb E\left[\left(a_1-a_3\right)^4\right]\leq 16\mathbb E\left[\left(X_n-\mu\right)^4\right]:
\end{IEEEeqnarray} 
\begin{IEEEeqnarray*}{rCl}
\mathbb E\left[B^2c_1^2\right] &\leq& C_0\sup_{A \leq t,s\leq 2A} \left[\left(\frac{\partial^2 f}{\partial x^2}\bigg|_{(s,t)}\right)^2+\left(\frac{\partial^2 f}{\partial x \partial y}\bigg|_{(s,t)}\right)^2\right]\mathbb E\left[\left(a_1-a_3\right)^4\right] \\
&\stackrel{(*)}{\leq} & C_1A^2\sup_{A \leq t,s \leq 2A} \left[\left(\frac{\partial^2 f}{\partial x^2}\bigg|_{(s,t)}\right)^2+\left(\frac{\partial^2 f}{\partial x \partial y}\bigg|_{(s,t)}\right)^2\right] \\
& = & {\scriptscriptstyle\mathcal{O}}\left(2^n\right)={\scriptscriptstyle\mathcal{O}}\left(1\right)\Var \left[X_n\right].
\end{IEEEeqnarray*} 
The inequality marked by $(*)$ follows from the fourth moment bound (see equation (\ref{eq:fourth_moment})).
Thus, if we combine the last inequality with equations (\ref{eq:variance_formula}) and (\ref{eq:variance_of_f}) we obtain:
\begin{IEEEeqnarray*}{rCl}
\Var \left[X_{n+1}\right]=\left(2+C_x^2+C_y^2+{\scriptscriptstyle\mathcal{O}}\left(1\right)\right)\Var \left[X_n\right],
\end{IEEEeqnarray*}
which completes the proof of the first part.

In order to show the second part we write the recursion relation (here we use the notation $\mu_n=\mathbb E\left[X_n\right]$):

\begin{IEEEeqnarray*}{rCl}
\frac{X_{n+1}-\mu_{n+1}}{\sqrt{\Var \left[X_{n+1}\right]}} &=& \frac{X_n-\mu_n}{\sqrt{\Var \left[X_n\right]}}\sqrt{\frac{\Var \left[X_n\right]}{\Var \left[X_{n+1}\right]}}+\frac{X_n^{'}-\mu_n}{\sqrt{\Var \left[X_n\right]}}\sqrt{\frac{\Var \left[X_n\right]}{\Var \left[X_{n+1}\right]}}+\\
&& \frac{f\left(X_n^{''},X_n^{'''}\right)-\mathbb E\left[f\left(X_n^{''},X_n^{'''}\right)\right]}{\sqrt{\Var \left[X_n\right]}}\sqrt{\frac{\Var \left[X_n\right]}{\Var \left[X_{n+1}\right]}} \nonumber
\end{IEEEeqnarray*} 

We plug into this the Taylor expansion of $f$ (here $(a,b)$ is in the convex hull of $\left(X_n^{''},X_n^{'''}\right)$ and $\left(\mu_n,\mu_n\right)$):
 
\begin{IEEEeqnarray*}{rCl}
f\left(X_n^{''},X_n^{'''}\right)&=&f\left(\mu_n,\mu_n\right)+\frac{\partial f}{\partial x}\bigg|_{(\mu_n,\mu_n)}\left(X_n^{''}-\mu_n\right)+\frac{\partial f}{\partial y}\bigg|_{(\mu_n,\mu_n)}\left(X_n^{'''}-\mu_n\right)\\
&& \text{ }+\text{ }\frac{1}{2}\frac{\partial^2 f}{\partial x^2}\bigg|_{(a,b)}\left(X_n^{''}-\mu_n\right)^2+\frac{1}{2}\frac{\partial^2 f}{\partial y^2}\bigg|_{(a,b)}\left(X_n^{'''}-\mu_n\right)^2\\
&& \text{ }+\text{ }\frac{1}{2}\frac{\partial^2 f}{\partial x\partial y}\bigg|_{(a,b)}\left(X_n^{''}-\mu_n\right)\left(X_n^{'''}-\mu_n\right)\,. 
\end{IEEEeqnarray*} 

Now we let $Y_n=\frac{X_n-\mu_n}{\sqrt{\Var \left[X_n\right]}}$ and apply the recursion relation $m$  times to obtain (here $Y_{n-m}^{(i)}$ are independent copies of $Y_{n-m}$):

\begin{IEEEeqnarray*}{rCl}
Y_n=\sum_{i=1}^{4^m}a_{i,n,m}Y^{(i)}_{n-m}+A_{n,m} 
\end{IEEEeqnarray*} 

By condition 1 and the first part of the theorem we obtain that $\sum_{i=1}^{4^m}a^2_{i,n,m}\rightarrow 1$ as $n \rightarrow \infty$. By condition 3 we can show that $\mathbb E\left[\lvert A_{m,n}\rvert\right]\leq \epsilon(n)g(m)$ where $\epsilon(n)$ goes to zero as $n$ grows and $g(m)$ grows in $m$. Thus if $Z$ is a standard normal random variable and $Z_{m,n}$ is a normal random variable with mean zero and variance $\sum_{i=1}^{4^m}a^2_{i,n,m}$ then:
 for each $\delta>0$
\begin{IEEEeqnarray*}{rCl}
\mathbb P\left(\lvert Y_n-Z \rvert > \delta\right)&\leq &  \mathbb P\left(\lvert \sum_{i=1}^{4^m}a_{i,n,m}Y^{(i)}_{n-m}-Z_{m,n} \rvert > \frac{\delta}{3}\right)+\mathbb P\left(\lvert Z_{m,n}-Z \rvert > \frac{\delta}{3}\right)\\
&& \text{ }+\text{ }\mathbb P\left(\lvert A_{m,n} \rvert > \frac{\delta}{3}\right)\\
&\leq & \varepsilon_1(m)+\varepsilon_2(n)+\mathbb E\left[\lvert A_{m,n}\rvert\right]\frac{3}{\delta}\\
&\leq & \varepsilon_1(m)+\varepsilon_2(n)+\epsilon_3(n)g(m)\frac{3}{\delta}
\end{IEEEeqnarray*} 

Here the notation $\varepsilon(n)$ means that $\varepsilon$ tends to zero as $n$ tends to infinity. Taking $n,m$ to infinity at the appropriate rate yields the result. Note that in order to bound $\mathbb P\left(\lvert \sum_{i=1}^{4^m}a_{i,n,m}Y^{(i)}_{n-m}-Z_{m,n} \rvert > \frac{\delta}{3}\right)$ we applied Lyaponov CLT with $\delta=2$ which can be used since 

\begin{IEEEeqnarray*}{rCl}
\frac{\sum_{i=1}^{4^m}a^4_{i,n,m}}{\left(\sum_{i=1}^{4^m}a^2_{i,n,m}\right)^2}\leq \frac{\max_i a^2_{i,n,m}}{\sum_{i=1}^{4^m}a^2_{i,n,m}}=\frac{\Pi _{k=n-m}^{n-1}\frac{\Var \left[ X_k\right]}{\Var \left[X_{k+1}\right]}}{\sum_{i=1}^{4^m}a^2_{i,n,m}}\rightarrow 0
\end{IEEEeqnarray*} 

as $m$ tends to infinity (the last equality holds for $n-m$ sufficiently large as can be proven by induction for fixed $n$ and increasing $m$) and in addition $\mathbb E\left[Y^4_{n-m}\right]$ is bounded by \ref{eq:fourth_moment_0} and \ref{4-moment-easy}.

The third part follows from Jensen's inequality. The limit, which lies between $1$ and $2$ by the remark at the beginning of the proof, can be approximated using Taylor expansion (see examples of specific functions below).

\end{proof}

\section*{Analysis of specific functions}

We start with the function $f_1(t,s)=\frac{ts}{t+s}\,$. 
$X_n$ has a probability interpretation in this case - it denotes the effective resistance between the endpoints of the "line-circle-line" sequence of graphs, whose definition appears in the introduction to this paper. 

For $f_1$ we have $\frac{\partial f_1}{\partial t}=\frac{s^2}{(t+s)^2}$ and $\frac{\partial f_1}{\partial s}=\frac{t^2}{(t+s)^2}\,$, and thus $C_x=C_y=\frac{1}{4}$ and $f_1(t,t)=\frac{t}{2}=\left(C_x+C_y\right)t\,$. We need one of the next two lemmas.

\begin{lem}\label{square_bound}
For all positive numbers $a_1$,$a_2$,$a_3$ and $a_4$ satisfying $1/2\leq\frac{a_i}{a_j}\leq 2$ for $i,j\in\{1,2,3,4\}$ we have
\begin{IEEEeqnarray*}{rCl}
\left(f_1\left(a_1,a_2\right)-f_1\left(a_3,a_4\right)\right)^2\leq \frac{20}{81}\left(\left(a_1-a_3\right)^2+\left(a_2-a_4\right)^2\right).
\end{IEEEeqnarray*}
\end{lem}
\begin{proof}
By first order Taylor expansion:
\begin{IEEEeqnarray*}{rCl}
f_1\left(a_1,a_2\right)-f_1\left(a_3,a_4\right) & = & \frac{\partial f_1}{\partial x}\bigg|_{\left(b_1,b_2\right)} \left(a_1-a_3\right)+\frac{\partial f_1}{\partial y}\bigg|_{\left(b_1,b_2\right)}\left(a_2-a_4\right)\\
&=& \left( \frac{b_2}{b_1+b_2}\right)^2\left(a_1-a_3\right)+\left( \frac{b_1}{b_1+b_2}\right)^2\left(a_2-a_4\right) .
\end{IEEEeqnarray*}
Here $\left(b_1,b_2\right)$ is in the convex hull of $\left(a_1,a_2\right)$ and $\left(a_3,a_4\right)\,$. If we write $b_i=\lambda a_i+\left(1-\lambda\right)a_{i+2}$ for $i=1,2$ then:
\begin{IEEEeqnarray*}{rCl}
b_1=\lambda a_1+\left(1-\lambda\right)a_3\leq \lambda 2a_2+\left(1-\lambda\right)2a_4=2b_2
\end{IEEEeqnarray*}
and similarly $b_2\leq 2b_1\,$. Thus, if we also apply the convexity of $g\left(x\right)=x^2$ we obtain:
\begin{IEEEeqnarray*}{rCl}
\left(f_1\left(a_1,a_2\right)-f_1\left(a_3,a_4\right)\right)^2 & = & \left(\left( \frac{b_2}{b_1+b_2}\right)^2\left(a_1-a_3\right)+\left( \frac{b_1}{b_1+b_2}\right)^2\left(a_2-a_4\right)\right)^2\\
& = & \frac{\left(b_1^2+b_2^2\right)^2}{\left(b_1+b_2\right)^4} \left(\frac{b_2^2}{b_1^2+b_2^2}\left(a_1-a_3\right)+\frac{b_1^2}{b_1^2+b_2^2}\left(a_2-a_4\right)\right)^2 \\
& \leq & \frac{b_1^2+b_2^2}{\left(b_1+b_2\right)^2}\times \\
&& \left(\frac{b_2^2}{\left(b_1+b_2\right)^2}\left(a_1-a_3\right)^2+\frac{b_1^2}{\left(b_1+b_2\right)^2}\left(a_2-a_4\right)^2\right) .
\end{IEEEeqnarray*}
Let $x=\frac{b_1}{b_1+b_2}\,$. The inequality $\frac{b_i}{b_j}\leq 2$ for $i,j\in\{1,2\}$ implies $\frac{1}{3}\leq\ x \leq \frac{2}{3}\,$. Thus $\frac{b_1^2+b_2^2}{\left(b_1+b_2\right)^2}=x^2+(1-x)^2$. Optimizing over the domain of $x$ yields $\frac{b_1^2+b_2^2}{\left(b_1+b_2\right)^2}\leq \frac{5}{9}\,$. Plugging in the bounds completes the first part of proof. 
\end{proof}

We can achieve a better bound than the one in the second part of Lemma \ref{square_bound} by a stronger optimization argument.

\begin{lem}\label{better_square_bound}
For all positive numbers $a_1$,$a_2$,$a_3$ and $a_4$ satisfying $1/2\leq\frac{a_i}{a_j}\leq 2$ for $i,j\in\{1,2,3,4\}$ we have
\begin{IEEEeqnarray*}{rCl}
\left(f_1\left(a_1,a_2\right)-f_1\left(a_3,a_4\right)\right)^2\leq \frac{17}{81}\left(\left(a_1-a_3\right)^2+\left(a_2-a_4\right)^2\right).
\end{IEEEeqnarray*}
\end{lem}
\begin{proof}
The inequality obviously holds whenever $a_1=a_3$ and $a_2=a_4\,$. Thus, we can write $\left(a_1-a_3\right)=m\left(a_2-a_4\right)$ or $\left(a_2-a_4\right)=m\left(a_1-a_3\right)$ for some $\left|m\right|\leq 1\,$. We may assume that $\left(a_2-a_4\right)=m\left(a_1-a_3\right)$ (if not since $f_1(t,s)=f_1(s,t)$ we can replace $a_1\leftrightarrow a_2$ and $a_3\leftrightarrow a_4$). Thus,
\begin{IEEEeqnarray*}{rCl}
\left(f_1\left(a_1,a_2\right)-f_1\left(a_3,a_4\right)\right)^2 & = & \left(\left( \frac{b_2}{b_1+b_2}\right)^2\left(a_1-a_3\right)+\left( \frac{b_1}{b_1+b_2}\right)^2\left(a_2-a_4\right)\right)^2\\
&=& \left(\left( \frac{b_2}{b_1+b_2}\right)^2+m\left( \frac{b_1}{b_1+b_2}\right)^2\right)^2\left(a_1-a_3\right)^2\\
& = & \left(\left( \frac{b_2}{b_1+b_2}\right)^2+m\left( \frac{b_1}{b_1+b_2}\right)^2\right)^2\times \\
&& \frac{1}{1+m^2}\left(\left(a_1-a_3\right)^2+\left(a_2-a_4\right)^2\right) .
\end{IEEEeqnarray*}
It suffices to take $0\leq m \leq 1$ since 
\begin{IEEEeqnarray*}{rCl}
\left(\left( \frac{b_2}{b_1+b_2}\right)^2+m\left( \frac{b_1}{b_1+b_2}\right)^2\right)^2 & \leq & \left(\left( \frac{b_2}{b_1+b_2}\right)^2+\left|m\right|\left( \frac{b_1}{b_1+b_2}\right)^2\right)^2.
\end{IEEEeqnarray*}
Let $x=\frac{b_2}{b_1+b_2}\,$. Then, as shown in the previous lemma, $x\in\left[\frac{1}{3},\frac{2}{3}\right]\,$. Since for each $m$ the function $\varphi_m\left(s\right)=s^2+m\left(1-s\right)^2$ is bounded by $4/9+\left(1/9\right)m$ whenever $1/3\leq s\leq 2/3\,$, we conclude that:
\begin{IEEEeqnarray*}{rCl}
\left(f_1\left(a_1,a_2\right)-f_1\left(a_3,a_4\right)\right)^2 & \leq & \frac{1}{81}\frac{\left(4+m\right)^2}{1+m^2}\left(\left(a_1-a_3\right)^2+\left(a_2-a_4\right)^2\right).
\end{IEEEeqnarray*}
Note that the function $\phi\left(m\right)=\frac{\left(4+m\right)^2}{1+m^2}$ achieves it's maximum in the interval $[0,1]$ at $m=1/4$ which completes the proof. The bound is tight, which we can see if we set $a_2=2\,$, $a_4=2-\varepsilon/4\,$, $a_1=1+\varepsilon$ and $a_3=1$ and let $\varepsilon\searrow 0\,$.
\end{proof}
If we apply (for instance) Lemma \ref{better_square_bound} we obtain $A=B=\frac{17}{81}$ and thus $A+B=\frac{34}{81}<\frac{1}{2}=C_x+C_y$. The other conditions are easy to verify and follow from the second moment terms $\frac{\partial ^2 f_1}{\partial t^2}=(-2)\frac{s^2}{(t+s)^3}\,$, $\frac{\partial ^2 f_1}{\partial t \partial s}=2\frac{ts}{(t+s)^3}$ and $\frac{\partial ^2 f_1}{\partial s^2}=(-2)\frac{t^2}{(t+s)^3}\,$.

\section*{Expectation bounds for $X_n$ for the choice $f_1=\frac{ts}{t+s}$}
We show lower and upper expectation bounds for $X_n$.
The upper bound follows from applying to the recurrence relation the fact that the harmonic mean is bounded above by the arithmetic mean (or by concavity of $f$). This yields $\mathbb E\left(X_{n+1}\right)\leq 2.5\mathbb E\left(X_n\right)$ and also that $\frac{1}{2.5^n}\mathbb E\left(X_n\right)$ converges. Now if one calculates the value of $\mathbb E\left(X_i\right)$ for some $i$ explicitly one obtains an upper bound. The simple calculation of $\mathbb E\left(X_0\right)$ yields $\mathbb E\left(X_n\right)\leq 1.5\times 2.5^n$. The calculation of $\mathbb E\left(X_1\right)$ only provides a slight improvement: $\mathbb E\left(X_n\right)\leq \frac{89}{60}\times 2.5^n$. For the lower bound we note that $\left|\frac{\partial ^2 f_1}{\partial t^2}\right|=\left|2\frac{s^2}{(t+s)^3}\,\right|\leq \frac{8}{27}\frac{1}{2.5^n}$, $\left|\frac{\partial ^2 f_1}{\partial t \partial s}\right|=\left|2\frac{ts}{(t+s)^3}\right|\leq \frac{1}{4}\frac{1}{2.5^n}$ and $\left|\frac{\partial ^2 f_1}{\partial s^2}\right|=\left|2\frac{t^2}{(t+s)^3}\right|\leq \frac{8}{27}\frac{1}{2.5^n}$ whenever $2.5^n\leq s,t \leq 2\times2.5^n$. Thus, for every $2.5^n\leq s,t,a \leq 2\times2.5^n$:

\begin{IEEEeqnarray*}{rCl}
f_1\left(t,s\right)& \geq & f_1\left(b,b\right)+\frac{1}{4}\left(t-b\right)+\frac{1}{4}\left(s-b\right)\\
&& -\frac{4}{27}\frac{1}{2.5^n}\left(t-b\right)^2-\frac{4}{27}\frac{1}{2.5^n}\left(s-b\right)^2-\frac{1}{8}\frac{1}{2.5^n}\left|t-b\right|\left|s-b\right|. 
\end{IEEEeqnarray*}

Thus if we replace $t$ and $s$ by two independent copies of $X_n$ and $b$ by $\mathbb E\left[X_n\right]$ and apply Lemma \ref{better_square_bound} we obtain:

\begin{IEEEeqnarray*}{rCl}
\mathbb E\left[X_{n+1}\right]&\geq & 2.5\mathbb E\left[X_n\right]-\frac{91}{216}\frac{1}{2.5^n}\Var\left[X_n\right]\\
& \geq & 2.5\mathbb E\left[X_n\right]-\frac{91}{216}\frac{\left(2+\frac{17}{81}\right)^n}{2.5^n}\Var\left[X_0\right]\\
& \geq & 2.5^{n+1}\mathbb E\left[X_0\right]-\frac{91}{216}\Var\left[X_0\right]\sum_{k=0}^{n} 2.5^k \frac{\left(2+\frac{17}{81}\right)^{n-k}}{2.5^{n-k}}\\
& \geq & 2.5^{n+1}\left[\mathbb E\left[X_0\right]-\frac{91}{216}\frac{2.5}{2.5^2-2-\frac{17}{81}}\Var\left[X_0\right]\right]
\end{IEEEeqnarray*}

By plugging in the values of $\Var\left[X_0\right]$ and $\mathbb E\left[X_0\right]$ we obtain $\mathbb E\left[X_n\right] \geq 1.43 \times 2.5^n$.

\section*{Another example - $f_2(t,s)=t^{\frac{1}{2}}s^{\frac{1}{2}}$}
We now turn to $f_2(t,s)=t^{\frac{1}{2}}s^{\frac{1}{2}}$. As mentioned before, this example satisfies a variant of condition $2$ of our theorem presented in remark $4$ but does not satisfy condition $2$ itself. First, $\frac{\partial f_2}{\partial t}=\frac{1}{2}t^{-\frac{1}{2}}s^{\frac{1}{2}}$ and $\frac{\partial f_2}{\partial s}=\frac{1}{2}s^{-\frac{1}{2}}t^{\frac{1}{2}}$, and thus $C_x=C_y=\frac{1}{2}\,$. By the recurrence relation:

\begin{IEEEeqnarray*}{rCl}
\mathbb E \left[X_{n+1}\right] &=& 2\mathbb E \left[X_n\right]+\mathbb E \left[\sqrt{X_n}\sqrt{X'_n}\right] \\
&=& 2\mathbb E \left[X_n\right]+\left( \mathbb E \left[\sqrt{X_n}\right] \right)^2 \\
&\leq& 3\mathbb E \left[X_n\right] 
\end{IEEEeqnarray*}
and thus $\mathbb E \left[X_n\right]\leq 3^n \mathbb E \left[X_0\right]=1.5\times 3^n$ (and also $\mathbb E \left[X_n\right]\leq 3^{n-1} \mathbb E \left[X_1\right]\leq 1.49\times 3^n$) and also $\frac{1}{3^n}\mathbb E\left(X_n\right)$ converges. Recall also that, as shown in the main theorem, all independent copies of $X_n$ lie between $3^n$ and $2\times 3^n$.
Let $a_1,a_2,a_3$ and $a_4$ be i.i.d. distributed as $X_n$. By Taylor expansion:

\begin{IEEEeqnarray*}{rCl}
\left(f_2\left(a_1,a_2\right)-f_2\left(a_3,a_4\right)\right)^2&=&\left(\frac{\partial f_1}{\partial x}\bigg|_{\left(b_1,b_2\right)}\left(a_1-a_3\right)+\frac{\partial f_1}{\partial y}\bigg|_{\left(b_1,b_2\right)}\left(a_2-a_4\right)\right)^2\\
&=& \left(\frac{1}{2}b_1^{-\frac{1}{2}}b_2^{\frac{1}{2}}\left(a_1-a_3\right)+\frac{1}{2}b_1^{\frac{1}{2}}b_2^{-\frac{1}{2}}\left(a_2-a_4\right)\right)^2 \\
&=& \frac{1}{4}\left[\frac{b_2}{b_1}\left(a_1-a_3\right)^2+\frac{b_1}{b_2}\left(a_2-a_4\right)^2\right]\\
&& +\text{ }\frac{1}{4}\left(a_1-a_3\right)\left(a_2-a_4\right).
\end{IEEEeqnarray*}
(here $\left(b_1,b_2\right)$ belongs to the convex hull of $\left(a_1,a_2\right)$ and $\left(a_3,a_4\right)\,$.) By applying the same technique in the analysis of $f_1$ we obtain that $\frac{b_i}{b_j}\leq 2$ for $i,j\in\{1,2\}\,$. Thus,

\begin{IEEEeqnarray*}{rCl}
\left(f_2\left(a_1,a_2\right)-f_2\left(a_3,a_4\right)\right)^2&\leq&\frac{1}{4}\left[2\left(a_1-a_3\right)^2+2\left(a_2-a_4\right)^2+\left(a_1-a_3\right)\left(a_2-a_4\right)\right]\\
&\leq & \frac{5}{8}\left(a_1-a_3\right)^2+\frac{5}{8}\left(a_2-a_4\right)^2
\end{IEEEeqnarray*}
and
\begin{IEEEeqnarray*}{rCl}
\mathbb E \left[\left(f_2\left(a_1,a_2\right)-f_2\left(a_3,a_4\right)\right)^2\right]&=&
\frac{1}{4}\left[\mathbb E\left[\frac{b_2}{b_1}\left(a_1-a_3\right)^2\right]+\mathbb E\left[\frac{b_1}{b_2}\left(a_2-a_4\right)^2\right]\right]\\
&& +\text{ }\frac{1}{4}\mathbb E\left[a_1-a_3\right]\mathbb E\left[a_2-a_4\right]\\
&=& \frac{1}{4}\mathbb E\left[b_2\right]\mathbb E\left[b_1^{-1}\left(a_1-a_3\right)^2\right]\\
&& +\frac{1}{4}\mathbb E\left[b_1\right]\mathbb E\left[b_2^{-1}\left(a_2-a_4\right)^2\right]+\text{ }0\\
&\leq& \frac{1}{4}\times1.5\times 3^n\times \frac{1}{3^n}\mathbb E\left[\left(a_1-a_3\right)^2\right]\\
&& +\text{ }\frac{1}{4}\times 1.5\times 3^n\times \frac{1}{3^n}\mathbb E\left[\left(a_2-a_4\right)^2\right]\\
&=& \frac{3}{8}\mathbb E\left[\left(a_1-a_3\right)^2\right]+\frac{3}{8}\mathbb E\left[\left(a_2-a_4\right)^2\right].
\end{IEEEeqnarray*}
Thus, we can take $A=B=\frac{5}{8}$ and $A_1=B_1=\frac{3}{8}\,$. Since $A_1+B_1=\frac{3}{4}<1=C_x+C_y$ and $2+A^2+B^2<\left(2+A_1+B_1\right)^2$ the weaker version of condition 2 in the theorem holds by remark $4\,$. However, $A+B=\frac{5}{4}>1=C_x+C_y$ and thus condition $2$ itself does not hold for this choice of $f_2\,$. Condition 3 is straightforward to show and follows from the expressions of the second order derivatives ($\frac{\partial ^2 f_2}{\partial t^2}=-\frac{1}{4}t^{-\frac{3}{2}}s^{\frac{1}{2}}$, $\frac{\partial ^2 f_2}{\partial t \partial s}=\frac{1}{4}t^{-\frac{1}{2}}s^{-\frac{1}{2}}$ and $\frac{\partial ^2 f_2}{\partial s^2}=-\frac{1}{4}t^{\frac{1}{2}}s^{-\frac{3}{2}}$). Note also that the partial derivatives of $f_2$ are not bounded, unlike the partial derivatives of $f_1\,$. By plugging in the above variance bound one can show (in a similar fashion as was done for $f_1$) that $\mathbb E\left[X_n\right]\geq 3^n\left[\mathbb E\left[X_0\right]-\frac{\sqrt{1.5}}{4}\frac{3}{3^2-2-\frac{1.49}{2}}\Var\left[X_0\right]\right]$ and thus $\mathbb E\left[X_n\right]\geq 1.46\times 3^n$.

I thank Itai Benjamini, Gady Kozma, Elliot Paquette and Igor Shinkar for helpful discussions.

\end{document}